\documentclass[a4paper,12pt]{article}
\usepackage{amsmath,amsfonts,amssymb,amscd}
\usepackage{indentfirst,graphicx,epsfig}
\usepackage{graphicx,psfrag}
\usepackage{amsthm, amsfonts, color}
\usepackage[bookmarksnumbered, plainpages]{hyperref}
\usepackage[numbers,sort&compress]{natbib}
\input{epsf}

\title{Proper connection number and\\ 2-proper connection number of a
graph\thanks{Supported by NSFC No.11371205 and PCSIRT.} }
\author{\small{Fei Huang, Xueliang Li, Shujing Wang}\\
{\small  Center for Combinatorics and LPMC-TJKLC}\\
{\small Nankai University, Tianjin 300071, China}\\\makeatletter
{\small Email: huangfei06@126.com; lxl@nankai.edu.cn;
\newcommand\figcaption{\def\@captype{figure}\caption}
wang06021@126.com} }
\newcommand\tabcaption{\def\@captype{table}\caption}
\date{}\makeatother
\newtheorem{theorem}{Theorem}[section]
\newtheorem{defi}{Definition}[section]
\newtheorem{lemma}[theorem]{Lemma}
\newtheorem{pro}[theorem]{proposition}
\newtheorem{coro}[theorem]{Corollary}

\begin{document}

\maketitle

\begin{abstract}
A path in an edge-colored graph is called a proper path if no two
adjacent edges of the path are colored with one same color. An
edge-colored graph is called $k$-proper connected if any two
vertices of the graph are connected by $k$ internally pairwise 
vertex-disjoint proper paths in the graph. The $k$-proper connection
number of a $k$-connected graph $G$, denoted by $pc_k(G)$, is
defined as the smallest number of colors that are needed in order to
make $G$ $k$-proper connected. For $k = 1$, we write $pc(G)$ other
than $pc_1(G)$, and call it the proper connection number of $G$. In
this paper, we present an upper bound for the proper connection
number of a graph $G$ in terms of the minimum degree of $G$, and
give some sufficient conditions for a graph to have $2$-proper
connection number two. Also, we investigate the proper connection
numbers of dense graphs.
\\[2mm]



\end{abstract}

\section{Introduction}

In this paper we are concerned with simple connected finite graphs.
We follow the terminology and the notation of Bondy and Murty
\cite{Bondy}. For a graph $G=(V,E)$ and two disjoint subsets $X$ and
$Y$ of $V$, denote by $B_G[X,Y]$ the bipartite subgraph of $G$ with
vertex set $X\cup Y$ and edge set $E(X,Y)$, where $E(X,Y)$ is the
set of edges of $G$ that have one end in $X$ and the other in $Y$. A
graph is called pancyclic if it contains cycles of all lengths $r$
for $3\le r\le n$. The join $G_1\vee G_2 $ of two edge-disjoint
graphs $G_1$ and $G_2$ is obtained by adding edges from each vertex
in $G_1$ to every vertex in $G_2$.

An edge-coloring of a graph $G$ is an assignment $c$ of colors to
the edges of $G$, one color to each edge of $G$. If adjacent edges
of $G$ are assigned different colors by $c$, then $c$ is a
\emph{proper (edge-)coloring}. For a graph $G$, the minimum number of colors needed
in a proper coloring of $G$ is referred to as the \emph{chromatic
index} of {\it edge-chromatic number} of $G$ and denoted by $\chi'(G)$. A path in an edge-colored
graph with no two edges sharing the same color is called a
\emph{rainbow path}. An edge-colored graph $G$ is said to be
\emph{rainbow connected} if every pair of distinct vertices of $G$
is connected by at least one rainbow path in $G$. Such a coloring is
called a \emph{rainbow coloring} of the graph. For a connected graph $G$, 
the minimum number of colors needed in a rainbow coloring of $G$ is referred to as 
the \emph{rainbow connection number} of $G$ and denoted by $rc(G)$. 
The concept of rainbow coloring was first introduced by Chartrand et al.
in \cite{Chartrand}. In recent years, the rainbow coloring has been
extensively studied and a variety of nice results have been obtained, 
see \cite{Chandran,Chartrand2,Kri,Li,Li3} for examples. For more details
we refer to a survey paper \cite{Li1} and a book \cite{Li2}.

Inspired by the rainbow coloring and proper coloring of graphs, 
Andrews et al. \cite{Andrews} introduced the concept of proper-path coloring. 
Let $G$ be an edge-colored graph, where adjacent edges may be colored with the  
same color. A path in $G$ is called a \emph{proper path} if no two adjacent 
edges of the path are colored with a same color. An edge-coloring $c$ is a 
\emph{proper-path coloring} of a connected graph $G$ if every pair of 
distinct vertices $u,v$ of $G$ is connected by a proper $u-v$ path in 
$G$. A graph with a proper-path coloring is said to be \emph{proper connected}. 
If $k$ colors are used, then $c$ is referred to as a \emph{proper-path 
$k$-coloring}. An edge-colored graph $G$ is {\it $k$-proper connected} if any 
two vertices are connected by $k$ internally pairwise vertex-disjoint 
proper paths. For a $k$-connected graph $G$, the \emph{$k$-proper connection number}
of $G$, denoted by $pc_k(G)$, is defined as the smallest number 
of colors that are needed in order to make $G$ $k$-proper connected. 
Clearly, if a graph is $k$-proper connected, then it is also $k$-connected. 
Conversely, any $k$-connected graph has an edge-coloring that makes it 
$k$-proper connected; the number of colors is easily bounded by the 
edge-chromatic number which is well known to be at most $\Delta(G)$ 
or $\Delta(G)+1$ by Vizing¡¯s Theorem \cite{Vizing} (where $\Delta(G)$, 
or simply $\Delta$, is the maximum degree of $G$). Thus $pc_k(G)\le \Delta(G)+ 1$ 
for any $k$-connected graph $G$. For $k = 1$, we write $pc(G)$ other than $pc_1(G)$, and 
call it the \emph{proper connection number} of $G$.

Let $G$ be a nontrivial connected graph of order $n$ and size $m$.
Then the proper connection number of $G$ has the following apparent bounds:
$$1 \le pc(G) \le \min\{\chi'(G), rc(G)\}\le m.$$ Furthermore, $pc(G)=1$ if and only 
if $G = K_n$ and $pc(G)=m$ if and only if $G=K_{1,m}$ is a star of size $m$.

 The paper is organized as follows: In Section 2, we give the basic definitions and 
 some useful lemmas. In Section 3, we study the proper connection number of bridgeless 
 graphs, and present a tight upper bound for the proper connection number of a 
 graph by using this result. In Section 4, we give some sufficient conditions for 
 graphs to have $2$-proper connection number two. In Section 5, we investigate the 
 proper connection number of dense graphs.

\section{Preliminaries}

At the beginning of this section, we list some fundamental results on proper-path coloring.

\begin{lemma}\label{lem2.1}\cite{Andrews} If $G$ is a nontrivial connected graph and $H$ 
is a connected spanning subgraph of $G$, then $pc(G)\le pc(H)$. In particular, 
$pc(G)\le pc(T)$ for every spanning tree $T$ of $G$.
\end{lemma}

\begin{lemma}\label{lem2.2}\cite{Andrews}
If $T$ is a nontrivial tree, then $pc(T)=\chi'(T)=\Delta(T)$.
\end{lemma}

Given a colored path $P =v_1v_2\ldots v_{s-1}v_s$ between any two vertices 
$v_1$ and $v_s$, we denote by $start(P)$ the color of the first edge in the path, 
i.e., $c(v_1v_2)$, and by $end(P)$ the last color, i.e., $c(v_{s-1}v_s)$. 
If $P$ is just the edge $v_1v_s$, then $start(P)=end(P)=c(v_1v_s)$.

\begin{defi}
Let $c$ be an edge-coloring of $G$ that makes $G$ proper connected. We say that 
$G$ has the strong property under $c$ if for any pair of vertices $u, v\in V(G)$, 
there exist two proper paths $P_1$, $P_2$ between them (not necessarily disjoint) 
such that $start(P_1)\neq start(P_2)$ and $end(P_1)\neq end(P_2)$.
\end{defi}

In \cite{Borozan}, the authors studied the proper-connection numbers in $2$-connected 
graphs. Also, they presented a result which improves the upper bound $\Delta(G)+1$ of 
$pc(G)$ to the best possible whenever the graph $G$ is bipartite and $2$-connected.

\begin{lemma}\cite{Borozan}\label{lem2.3}
Let $G$ be a graph. If $G$ is bipartite and $2$-connected, then $pc(G)=2$ and there 
exists a $2$-edge-coloring $c$ of $G$ such that $G$ has the strong property under $c$.
\end{lemma}

\begin{lemma}\cite{Borozan}\label{lem2.4}
Let $G$ be a graph. If $G$ is $2$-connected, then $pc(G)\leq 3$ and there exists a 
$3$-edge-coloring $c$ of $G$ such that $G$ has the strong property under $c$.
\end{lemma}

\begin{lemma}\cite{HLW}\label{lem2.5}
Let $H=G\cup\{v_1\}\cup\{v_2\}$. If there is a proper-path $k$-coloring $c$ of $G$ such 
that $G$ has the strong property under $c$. Then $pc(H)\leq k$ as long as $v_1,v_2$ are 
not isolated vertices of $H$.
\end{lemma}
As a result of Lemma \ref{lem2.5}, we obtain the following corollary.
\begin{coro}\label{cor2.6}
Let $H$ be the graph that is obtained by identifying $u_i$ of $G$ to $v_i$ of a path $P^i$ 
for $i=1,2$, where $d_{P^i}(v_i)=1$. If there is a proper-path $k$-coloring $c$ of $G$ such 
that $G$ has the strong property under $c$, then $pc(H)\leq k$.
\end{coro}

\begin{lemma}\label{lem4.2}
Let $C_n=v_1v_2\ldots v_nv_1$ be an $n$-vertex cycle and $G=C_n+v_{n-1}v_1$. One has that $pc_2(G)=2$.
\end{lemma}
\begin{proof}
If $n$ is an even integer, it is obvious that $pc_2(G)\leq pc_2(C_n)=2$. So we only need to consider the 
case that $n$ is odd. Let $C'=v_1v_2v_3\ldots v_{n-1}v_1$. Then we have that $C'$ is an even cycle. We 
color the edges $v_{2i-1}v_{2i}$ by color 1 for $i=1, 2, \ldots, \frac{n-1}{2}$ and color the other edges 
by color 2. Now we show that for any $v_i, v_j$, there are two disjoint proper paths connecting them. 
If $i, j\neq n$, we can see that $P_1=v_iv_{i+1}\ldots v_j$ and $P_2=v_iv_{i-1}\ldots v_{1}v_{n-1}\ldots v_j$ 
are two disjoint proper paths connecting $v_i$ and $v_j$. If $i=n$, we also have that $Q_1=v_nv_1v_2\ldots v_j$ 
and $Q_2=v_nv_{n-1}v_{n-2}\ldots v_j$ are two disjoint proper paths connecting $v_n$ and $v_j$. The proof is thus 
complete.
\end{proof}

\section{Upper bounds of proper connection number}

In \cite{Borozan}, the authors studied the proper connection number for $2$-(edge)-connected graphs by 
(closed) ear-decomposition. Here, we reproof the result for $2$-edge-connected graphs by using another method.

\begin{theorem}\label{th2.1}
If $G$ is a connected bridgeless graph with $n$ vertices, then $pc(G)\le 3$. Furthermore, there exists a 
$3$-edge-coloring $c$ of $G$ such that $G$ has the strong property under $c$.
\end{theorem}
\begin{proof}
We prove the result for connected bridgeless graphs by induction on the number of blocks in $G$. First, 
the result clearly holds when $G$ is 2-connected by Lemma \ref{lem2.4}. Suppose that $G$ has at least 
two blocks. Let $X$ be the set of vertices of an end-block of $G$, that is, $X$ contains only one cut 
vertex, say $x$. From Lemma \ref{lem2.4}, we know that $G[X]$ has a $3$-edge-coloring $c_1$ such that 
$G[X]$ has the strong property under $c_1$. Consider the subgraph $H$ of $G$ induced by $(V(G)\setminus X)\cup\{x\}$. 
$H$ is a connected bridgeless graph with the number of blocks one less than $G$. By the induction hypothesis, 
we have that $pc(H)\le 3$ and $H$ has a $3$-edge-coloring $c_2$ such that $H$ has the strong property 
under $c_2$. Let $c$ be the edge-coloring of $G$ such that $c(e)=c_1(e)$ for any $e\in E(G[X])$ and $c(e)=c_2(e)$ 
otherwise. We now show that $G$ has the strong property under the coloring $c$. It suffices to consider the 
pairs $u, v$ such that $u\in X\setminus\{x\}$ and $v\in V(G)\setminus X$. Let $P_1$, $P_2$ be two proper 
paths in $G[X]$ between $u$ and $x$ such that $start(P_1)\neq start(P_2)$ and $end(P_1)\neq end(P_2)$, and let 
$Q_1, Q_2$ be the two proper paths in $H$ between $v$ and $x$ such that $start(Q_1)\neq start(Q_2)$ and 
$end(Q_1)\neq end(Q_2)$. It is obvious that either $P=P_1\cup Q_1, Q=P_2\cup Q_2$ or $P=P_1\cup Q_2, Q=P_2\cup Q_1$ 
are two proper paths between $u$ and $v$ with the property that $start(P)\neq start(Q)$ and $end(P)\neq end(Q)$. 
This completes the proof.
\end{proof}
With a similar analysis and by Lemma \ref{lem2.3}, we have the following theorem.
\begin{theorem}
If $G$ is a bipartite connected bridgeless graph with $n$ vertices, then $pc(G)\le 2$. 
Furthermore, there exists a $2$-edge-coloring $c$ of $G$ such that $G$ has the strong property under $c$.
\end{theorem}
An Eulerian graph is clearly bridgeless. As a result of Theorem \ref{th2.1}, we have the following corollary.
\begin{coro}
For any Eulerian graph $G$, one has that $pc(G)\leq 3$. Furthermore, if $G$ is Eulerian and bipartite, 
one has that $pc(G)= 2$.
\end{coro}
\begin{lemma}\label{lem3.2}
Let $G$ be a graph and $H=G-PV(G)$, where $PV(G)$ is the set of the pendent vertices of $G$. 
Suppose that $pc(H)\le 3$ and there is a proper-path $3$-coloring $c$ of $H$ such that $H$ has the 
strong property under $c$. Then one has that $pc(G)\le \max\{3,|PV(G)|\}$.
\end{lemma}
\begin{proof}
Assume that $PV(G)=\{v_1,v_2,\ldots,v_k\}$. If $k\leq 2$, we have that $pc(G)\le 3$ by Lemma \ref{lem2.5}. 
So we consider the case that $k\geq 3$. Let $u_i$ be the neighbor of $v_i$ in $G$ for $i=1,2,\ldots, k$, 
and let $\{1,2,3\}$ be the color-set of $c$. We first assign color $j$ to $u_jv_j$ for $j=4,\ldots, k$. 
Then we color the remaining edges $u_1v_1, u_2v_2, u_3v_3$ by colors $1,2,3$.

If $u_1=u_2=u_3$, we assign color $i$ to $u_iv_i$ for $i=1,2,3$. If $u_1=u_2\neq u_3$, let $P$ be a proper 
path of $G$ connecting $u_1$ and $u_3$. Then there are two different colors in $\{1,2,3\}\setminus \{start(P)\}$. 
We assign these two colors to $u_1v_1$ and $u_2v_2$, respectively, and choose a color that is distinct from 
$end(P)$ in \{1,2,3\} for $u_3v_3$. If $u_i\neq u_j$ for $1\leq i\neq j\leq 3$, suppose that $P_{ij}$ is a 
proper path of $G$ between $u_i$ and $u_j$. We choose a color that is distinct from $start(P_{12})$ and 
$start(P_{13})$ in \{1,2,3\} for $u_1v_1$. Similarly, we color $u_2v_2$ by a color in 
$\{1,2,3\}\setminus \{end(P_{12}),start(P_{23})\}$, and color $u_3v_3$ by a color in 
$\{1,2,3\}\setminus\{end(P_{13}), end(P_{23})\}$.

One can see that in all these cases, $v_i$ and $v_j$ are proper connected for $1\le i\neq j\le k$. 
Moreover, as $H$ has the strong property under edge-coloring $c$, it is obvious that $v_i$ and $u$ 
are proper connected for $1\le i\le k$ and $u\in V(H)$. Therefore, we have that $pc(H)\leq k=|PV(G)|$.
Hence, we obtain that $pc(G)\le \max\{3,|PV(G)|\}$.
\end{proof}

\begin{lemma}\label{le3.3}
Let $G$ be a graph with a cut-edge $v_1v_2$, and $G_i$ be the connected graph obtained from $G$ by 
contacting the connected component containing $v_i$ of $G-v_1v_2$ to a vertex $v_i$, where $i=1,2$. 
Then $pc(G)=\max\{pc(G_1),pc(G_2)\}$
\end{lemma}
\begin{proof}
First, it is obvious that $pc(G)\ge \max\{pc(G_1),pc(G_2)\}$. Let $pc(G_1)=k_1$ and $pc(G_2)=k_2$. 
Without loss of generality, suppose $k_1\ge k_2$.  Let $c_1$ be a $k_1$-proper coloring of $G_1$ 
and $c_2$ be a $k_2$-proper coloring of $G_2$ such that $c_1(v_1v_2)=c_2(v_1v_2)$ and 
$\{c_2(e):e\in E(G_2)\}\subseteq \{c_1(e):e\in E(G_1)\}$. Let $c$ be the edge-coloring of $G$ such that 
$c(e)=c_1(e)$ for any $e\in E(G_1)$ and $c(e)=c_2(e)$ otherwise. Then $c$ is an edge-coloring of $G$ using $k_1$ 
colors. We will show that $c$ is a proper-path coloring of $G$. For any pair of vertices $u, v\in V(G)$, 
we can easily find a proper path between them if $u,v\in V(G_1)$ or $u,v\in V(G_2)$. Hence we only need 
to consider that $u\in V(G_1)\setminus\{v_1,v_1\}$ and $v\in V(G_2)\setminus\{v_1,v_2\}$. Since $c_1$ is 
a $k_1$-proper coloring of $G_1$, there is a proper path $P_1$ in $G_1$ connecting $u$ and $v_1$.  
Since $c_2$ is a $k_2$-proper coloring of $G_2$, there is a proper path $P_2$ in $G_2$ connecting 
$v$ and $v_2$. As $c_1(v_1v_2)=c_2(v_1v_2)$, then we know that $P=uP_1v_2v_1P_2v$ is a proper path 
connecting $u$ and $v$ in $G$. Therefore, we have that $pc(G)\le k_1$, and the proof is thus complete.
\end{proof}

Let $B\subseteq E$ be the set of cut-edges of a graph $G$.  Let $\mathcal{C}$ denote the 
set of connected components of $G'= (V;E\setminus B)$. There are two types of elements in $\mathcal{C}$, 
singletons and connected bridgeless subgraphs of $G$. Let $\mathcal{S}\subseteq\mathcal{C}$ denote 
the singletons and let $\mathcal{D}=\mathcal{C}\setminus\mathcal{S}$. Each element of $\mathcal{S}$ is, therefore, 
a vertex, and each element of $\mathcal{D}$ is a connected bridgeless subgraph of $G$.  Contracting each element 
of $\mathcal{D}$ to a vertex, we obtain a new graph $G^*$. It is easy to see that $G^*$ is a tree, and 
the edge set of $G^*$ is $B$. According to the above notations, we have the following theorem.

\begin{theorem}\label{th3.4}
If $G$ is a connected graph, then $pc(G)\le \max\{3, \Delta(G^*)\}$.
\end{theorem}
\begin{proof}
For an arbitrary element $A$ of $\mathcal{C}$, $A$ is either a singleton or a connected bridgeless subgraph  
of $G$. Let $C(A)$ be the set of cut-edges in $G$ that has an end-vertex in $A$. It is obvious that 
$|C(A)|\le \Delta(G^*)$. We use $A_0$ to denote the subgraph of $G$ obtained from $A$ by adding all the 
edges of $C(A)$ to $A$. If $A$ is a singleton, we have that $pc(A_0)=|C(A)|\le \max\{3, \Delta(G^*)\}$. 
Otherwise, from Theorem \ref{th2.1}, we know that $pc(A)\le 3$ and there is a coloring $c$ of $A$ such that 
$A$ has the strong property under $c$. Then by Lemma \ref{lem3.2}, we have that 
$pc(A_0)\le \max\{3, |C(A)|\}\le \max\{3, \Delta(G^*)\}$. Hence, by Lemma \ref{le3.3}, we can obtain 
that $pc(G)= \max_{A\in \mathcal{C}} pc(A_0)\le \max\{3, \Delta(G^*)\}$.
\end{proof}

Let $rK_t$ be the disjoint union of $r$ copies of the complete graph $K_t$, We use $S_r^t$ to denote 
the graph obtained from $rK_t$ by adding an extra vertex $v$ and joining $v$ to one vertex of each $K_t$.

\begin{coro}
If $G$ is a connected graph with $n$ vertices and minimum degree $\delta\ge 2$, then  $pc(G)\le \max\{3, \frac{n-1}{\delta+1}\}$. 
Moreover, if $\frac{n-1}{\delta+1} >3$, and $n\ge \delta(\delta+1)+1 $, we have that $pc(G)=\frac{n-1}{\delta+1}$ 
if and only if $G\cong S_r^t$, where $t-1=\delta$ and $rt+1=n$.
\end{coro}
\begin{proof}
Since the minimum degree of $G$ is $\delta\ge 2$, we know that each leaf of $G^*$ is obtained by contracting 
an element with at least $\delta+1$ vertices of $\mathcal{D}$. Therefore, $\mathcal{D}$ has at 
most $\frac{n-1}{\delta+1}$ such elements, and so, one can see that $\Delta(G^*)\leq  \frac{n-1}{\delta+1}$. 
From Theorem \ref{th3.4}, we know that $pc(G)\le \max\{3, \frac{n-1}{\delta+1}\}$.

If $\frac{n-1}{\delta+1} >3$ and $pc(G)=\frac{n-1}{\delta+1}$, one can see that $G^*$ is a star with 
$\Delta(G^*)=\frac{n-1}{\delta+1}$, and each leaf of $G^*$ is obtained by contracting an element with 
$\delta+1$ vertices of $\mathcal{D}$, that is, $G\cong S_r^t$,  where $t=\delta$ and $rt+1=n$. On the other 
hand, if $G\cong S_r^t$, where $t=\delta$ and $rt+1=n$, we can easily check that $pc(G)=r=\frac{n-1}{\delta+1}$.
 \end{proof}

\section{Graphs with $2$-proper connection number two}

At the beginning of this section, we list an apparent sufficient condition for graphs 
to have proper connection number two.

\begin{pro}\label{pro4.1}
 Let $G$ be a simple noncomplete graph on at least three vertices in which the minimum 
 degree is at least $n/2$, then $pc(G)=2$.
\end{pro}

We should mention that the condition $\delta(G)\ge n/2$ is quite rough. In \cite{Borozan}, 
the authors gave a much better result for graphs with appreciable quantity of vertices to 
have proper connection number two. They proved that if $G$ is connected noncomplete graph 
with $n\ge 68$ vertices, and $\delta\ge n/4$, then $pc(G)=2$.

It is easy to find that the $2$-proper connection number of any simple 2-connected graph is at least 2, 
and every complete graph on at least $4$ vertices evidently has the property that $pc_2(K_n)=2$. 
But suppose that our graph has considerably fewer edges. In particular, we may ask how large the minimum 
degree of $G$ must be in order to guarantee the property that $pc_2(G)=2$. Motivated by Proposition \ref{pro4.1}, 
we consider the condition $\delta \ge n/2$ and get the following theorem.

\begin{theorem}
Let $G$ be a connected graph with $n$ vertices and minimum degree $\delta$. If $\delta\ge n/2$ and $n\ge 4$, then $pc_2(G)=2$.
\end{theorem}

\begin{proof}
Since $\delta\ge n/2$, we know that there exists a Hamiltonian cycle $C=v_1v_2\ldots v_n$ in $G$. If $n$ is even, 
one has that $pc_2(G)\leq pc_2(C_n)=2$. Hence, we only need to consider the case that $n=2k+1$. Let $H=G-v_n$, 
one has that $d_H(v_i)\geq d_G(v_i)-1\geq k=|V(H)|/2$. Hence, there exists a Hamiltonian cycle $C'=v_1'v_2'\ldots v_{2k}'$ 
in $H$. As $d_G(v_n)\geq k+1$, one can see that there is an edge, say $v_1'v_2'$, such that $v_nv_1', v_nv_2'\in E(G)$. 
Hence, there is a spanning subgraph $G'$ of $G$ with $E(G')=E(C')\cup \{v_nv_1', v_nv_2'\}$. By Lemma \ref{lem4.2}, 
we have that $pc_2(G)\leq pc_2(G')=2$, and so the proof is complete.
\end{proof}

\noindent \textbf{Remark}: The condition $\delta\geq n/2$ is best possible. In fact, we can find graphs with minimum degree less 
than $n/2$ which is not $2$-connected, and so we cannot calculate the $2$-proper connection number. For example, 
let $G=K_1\vee(2K_k)$. We know that $\delta(G)=k<|V(G)|/2$, whereas $pc_2(G)$ does not exists.

Though the condition on the minimum degree can not be improved, we may consider some weaker conditions. We need an important 
conclusion which can be found in \cite{bb}. We use $cir(G)$ to denote the circumference (length of a longest cycle) of $G$.

\begin{lemma}\label{lem4.3}\cite{bb}
If $G=G(n,m)$ and $m\geq n^2/4$, then $G$ contains a cycle $C_r$ of length $r$ for each $3\le r\le cir(G)$.
\end{lemma}

\begin{theorem}
Let $G$ be a simple graph on at least four vertices in which the degree sum
of any two nonadjacent vertices is at least $n$. Then $pc_2(G)=2$.
\end{theorem}
\begin{proof}
Since the degree sum of any two nonadjacent vertices of $G$ is at least $n$, we know that $G$ is Hamiltonian. Suppose 
that a Hamiltonian cycle of $G$ is $C=v_1v_2\ldots v_n$. If $n$ is even or $n=5$, it is obvious that $pc_2(G)\leq 2$.

If $\delta\le3$, suppose, without loss of generality, that $v_n$ is the vertex which has minimum degree. There exists a 
$2<j<n-2, j\neq i$ such that $v_j$ and $v_n$ are not adjacent. We have that $d(v_j)\geq n-d(v_n)\ge n-3$, and so we know 
that $\{v_{j-2}, v_{j+2}\}\cap N(v_j)\neq \emptyset$, where the subscripts are modulus $n$. Hence, we can see that $pc_2(G)\leq 2$ 
by Lemma \ref{lem4.2}. In what follows of the proof, we only consider the case that $n$ is an odd number which is larger 
than 7 and $\delta\ge4$. To continue our proof, we need the following claim.

\noindent\textbf{Claim:} $G$ is pancyclic.

\noindent\emph{Proof of the Claim:} Let $\overline{E}$ be the edge set of $\overline{G}$. One can see that for any $uv\in \overline{E}$, 
$d_G(u)+d_G(v)\geq n$. It is obvious that
\begin{equation}\label{1}
\sum_{uv\in E\cup \overline{E}}(d_G(u)+d_G(v))=(n-1)\sum_{u\in V(G)}d_G(v)=2(n-1)m.
\end{equation}
One the other hand,
\begin{equation}\label{2}
 \sum_{uv\in E}(d_G(u)+d_G(v))=\sum_{u\in V(G)}d_G(u)^2\geq n(\sum_{u\in V(G)}d_G(u)/n)^2=4m^2/n,
\end{equation}
and
\begin{equation}\label{3}
 \sum_{uv\in \overline{E}}(d_G(u)+d_G(v))\geq ({n\choose 2}-m)n, 
\end{equation}
where the equality holds in (2) if and only if $G$ is a regular graph and the equality holds in (3) if and only if $d_G(u)+d_G(v)=n$ 
for each pair of nonadjacent vertices $u$ and $v$.

Thus we have 
\[
2(n-1)m\geq 4 m^2/n+({n\choose 2}-m)n,
\]
i.e.,
\[
(m-n^2/4)(m-{n\choose 2})\leq 0.
\]
Hence, we have that $m\geq n^2/4$, with equality holds if and only if $G$ is a regular graph with degree 
$n/2$. We know that $G$ is pancyclic from Lemma \ref{lem4.3}.

Assume that $C'=u_1u_2\ldots u_{n-1}$ is a cycle of $G$ with $n-1$ vertices and $v\not\in V(C')$. 
Without loss of generality, assume that $\{u_1, u_i, u_j, u_k\}\subseteq N(v)$ such that $1<i<j<k\le n-1$. 
Let $c(u_{2i-1}u_{2i})=1$ and $c(u_{2i}u_{2i+1})=2$ for $i=1,2,\ldots, \frac{n-1}{2}$, and let 
$c(vu_1)=c(u_{n-1}u_1), c(vu_i)=c(u_iu_{i+1}), c(vu_j)=c(u_{j-1}u_j), c(vu_k)=c(u_ku_{k+1})$. 
Now we prove that for any $x, y\in V(G)$, there are two disjoint proper paths connecting them. We only need 
to consider the case that $x=v, y=u_l\in V(C')$. If $1\leq l \leq k$, then $P_1=vu_1u_2\ldots u_l$ and 
$P_2=vu_ku_{k-1}\ldots u_l$ are two disjoint proper paths connecting them. If $k<l\leq n-1$, then 
$P_1=vu_iu_{i-1}\ldots u_1u_{n-1}\ldots u_l$ and $P_2=vu_ju_{j+1}\ldots u_l$ are two disjoint proper paths 
connecting them. Hence, we have that $pc_2(G)\leq 2$.
\end{proof}

\noindent \textbf{Remark}:
The condition that ``the degree sum of any two nonadjacent vertices of $G$ is at least $n$" cannot be improved. 
For example, $C_5$ and $K_1\vee(2K_k)$ have the property that the degree sum of any two nonadjacent vertices of 
is one less than their number of vertices, whereas $pc_2(C_5)=3$ and $pc_2(K_1\vee(2K_k)$ does not exist.

\section{Proper connection number of dense graphs}

In this section, we consider the following problem:

\noindent{\bf Problem 1.} For every $k$ with $1\le k\le n-1$, compute and minimize the function $f(n,k)$ with 
the following property: for any connected graph $G$ with $n$ vertices, if $|E(G)|\ge f(n,k)$, then $pc(G)\le k$.

In \cite{Kem}, this kind of question was suggested for rainbow connection number $rc(G)$, and in \cite{LLS}, 
the authors considered the case $k=3$ and $k=4$ for rainbow connection number $rc(G)$. We first show a lower bound 
for $f(n,k)$.
\begin{pro}
$f(n,k)\ge {n-k-1\choose2}+k+2.$
\end{pro}
\begin{proof}
We construct a graph $G_k$ as follows: Take a $K_{n-k-1}$ and a star $S_{k+2}$. Identify the center-vertex of 
$S_{k+2}$ with an arbitrary vertex of $K_{n-k-1}$. The resulting graph $G_k$ has order $n$ and size $E(G_k)=
{n-k-1\choose2}+k+1$. It can be easily checked that $pc(G_k)=k+1$. Hence, $f(n,k)\ge {n-k-1\choose2}+k+2$.
\end{proof}

\begin{lemma}\label{lem5.1}
Let $G$ be a graph with $n \ (n\geq 6)$ vertices and at least ${n-1\choose2}+3$ edges. Then for any $u, v\in V(G)$, 
there is a 2-connected bipartite spanning subgraph of $G$ with $u, v$ in the same part.
\end{lemma}
\begin{proof}
Let $\overline{G}$ be the complement of $G$. Then we have that $|E(\overline{G})|\leq n-4$. Let $S=N(u)\cap N(v)$, we have 
that $|S|\geq 2$. Since otherwise, $|S|\leq 1$, then one has that for any $w\in V(G)\setminus (S\cup \{u,v\})$,
either $uw\in E(\overline{G})$ or $vw\in E(\overline{G})$, and thus $|E(\overline{G})|\geq n-3$, which contradicts the fact 
that $|E(\overline{G})|\leq n-4$. Therefore, we know that $B_G[S, \{u,v\}]$ is a 2-connected bipartite subgraph of 
$G$ with $u, v$ in the same part.

Suppose that $H=B_G[X,Y]$ is a 2-connected bipartite subgraph of $G$ with $u, v$ in the same part 
and $H$ has as many vertices as possible. Then, if $V(G)\setminus V(H)\neq \emptyset$, one has that there exists a 
vertex $w\in V(G)\setminus V(H)$, such that $|N(w)\cap X|\geq 2$ or $|N(w)\cap Y|\geq 2$. Since otherwise,
$$|E(\overline{G})|\geq (n-|V(H_1)|)(|V(H_1)|-2)\geq n-3,$$
which contradicts the fact that $|E(\overline{G})|\leq n-4$. Then $w$ can be added to $X$ if $|N(w)\cap X|\geq 2$ or 
added to $Y$ otherwise, which contradicts the maximality of $H$. So, we know that $H$ is a 2-connected bipartite spanning 
subgraph of $G$ with $u, v$ in the same part, which completes the proof.
\end{proof}

\begin{lemma}\label{lem5.2}
Every 2-connected graph on $n$ $(n\ge 12)$ vertices with at least ${n-1\choose2}-5$ edges contains a 2-connected bipartite 
spanning subgraph.
\end{lemma}
\begin{proof}
The result is trivial if $G$ is complete. We will prove our result by induction on $n$ for noncomplete graphs. First, 
if $|V(G)|=12$ and $|E(G)|\geq 50$, one can find a $2$-connected bipartite spanning subgraph of $G$. So we suppose 
that the result holds for all 2-connected graphs on $n_0$ $(13<n_0<n)$ vertices with at least ${n_0-1\choose2}-5$ edges. 
For a 2-connected graph $G$ on $n$ vertices with $|E(G)|\ge {n-1\choose2}-5$, let $v$ be a vertex with minimum degree of $G$, 
and let $H=G-v$. If $d(v)=2$, then $|E(H)|\ge {n-1\choose2}-7$. Let $N_G(v)=\{v_1,v_2\}$. We know that $H$ contains a 2-connected 
bipartite spanning subgraph with $v_1,v_2$ in the same part by Lemma \ref{lem5.1}. Clearly, $G$ contains a 2-connected bipartite 
spanning subgraph. Otherwise, $3\le d(v)\le n-2$, then $|E(H)|\ge {n-1\choose2}-5-(n-2)={(n-1)-1\choose2}-5$ and $\delta(H)\ge 2$. 
If $H$ has a cut-vertex $u$, then each connected component of $H-u$ contains at least $2$ vertices. We have that $|E(H)|\le {n-3\choose2}+3<{n-2\choose2}-5$, 
a contradiction. Hence, $H$ is 2-connected. By the induction hypothesis, we know that $H$ contains a 2-connected bipartite spanning subgraph $B_H[X,Y]$. Since $d(v)\ge 3$, at least one of  $X$ and $Y$ contains at least $2$ neighbors of $v$. Hence, $G$ contains a 2-connected bipartite spanning subgraph.
\end{proof}

\begin{theorem}\label{th5.4}
Let $G$ be a connected graph of order $n\ge 14$. If ${n-3\choose 2}+4\le |E(G)|\le {n\choose 2}-1$, then $pc(G)=2$.
\end{theorem}
\begin{proof}
The result clearly holds if $G$ is $3$-connected. We only consider of the graphs with connectivity at most $2$.

\noindent\textbf{Claim 1:} $\delta(G)\le 5$.

\noindent\emph{Proof of Claim 1:} Suppose to the contrary that $\delta(G)\ge 6$.
If $G$ has a cut-vertex, say $x$, then each connected component of $G-x$ has at least $6$ vertices. Hence, we have 
that $|E(G)|\le {n-6\choose 2}+{7\choose 2}$, which is less than ${n-3\choose 2}+4$ when $n\ge 14$, a contradiction. 
If $G$ is 2-connected with a 2-vertex cut $\{x,y\}$, then each connected component of $G-x-y$ has at least $5$ vertices. 
We have that $|E(G)|\le {n-5\choose 2}+{7\choose 2}-1$, which is less than ${n-3\choose 2}+4$ when $n\ge 14$. 
We can also get a contradiction. Hence, we get our conclusion $\delta(G)\le 5$.

Let $v$ be a vertex with the minimum degree in $G$, and let $H=G-v$. Then $|V(H)|=n-1$ and $|E(H)|\ge {n-3\choose 2}+4-5={n-3\choose 2}-1$.

Note that if $H$ is $3$-connected, one can get that $pc(H)\leq 2$. Then by Lemma \ref{lem2.5}, one has that $pc(G)\leq 2$. 
So, we only consider the case that the connectivity of $H$ is at most $2$.

\noindent\textbf{Claim 2:} $\delta(H)\le 4$.

\noindent\emph{Proof of Claim 2:} We use the same method as in the proof of Claim 1. Suppose that $\delta(H)\ge 5$.
If $H$ has a cut-vertex, say $x$, then each connected component of $H-x$ has at least $5$ vertices. Hence, we have 
that $|E(H)|\le {n-6\choose 2}+{6\choose 2}$, which is less than ${n-3\choose 2}-1$ when $n\ge 14$, a contradiction. 
If $H$ is 2-connected with a 2-vertex cut $\{x,y\}$, then each connected component of $H-x-y$ has at least $4$ vertices. 
Hence, we have that $|E(H)|\le {n-5\choose 2}+{6\choose 2}-1$, which is less than ${n-3\choose 2}-1$ when $n\ge 14$. 
Hence we get our conclusion that $\delta(H)\le 4$.

Let $u$ be a vertex with the minimum degree in $H$, and let $F=H-u=G-v-u$. Then $|V(F)|=n-2$ and 
$|E(F)|\ge {n-3\choose 2}-5={(n-2)-1\choose 2}-5$. If $F$ is 2-connected, we know that $F$ contains a bipartite $2$-connected  
spanning subgraph by Lemma \ref{lem5.2}, and hence $pc(H)\leq 2$. By Lemma \ref{lem2.5}, we have that $pc(G)\le 2$. So, we only 
need to consider the case that $F$ has cut-vertices. As in the proof of Lemma \ref{lem5.2}, we know that $F$ has a pendent 
vertex $w$, and so $\delta(G)\le d_G(w)\le 3$. Let $F'=F-w=G-u-v-w$, then $|E(F')|\ge {n-3\choose 2}-6$. From Lemma \ref{lem5.1}, 
we know that $F'$ contains a 2-connected bipartite spanning subgraph, and so $pc(F')\leq 2$. If $d_G(w)=1$, then $u$ and $v$ are 
also pendent vertices in $G$. We have that $|E(G)|\le {n-3\choose 2}+3$, which contradicts the fact that $|E(G)|\ge {n-3\choose 2}+4$.
Thus, $d(w)\ge 2$. If $uv\in E(G)$, one can see that $pc(G)=2$ by Corollary \ref{cor2.6}. If $uv\not\in E(G)$, we have that $u$ has 
a neighbor in $F'$. Since otherwise, $d(u)=1$ and $d(v)=1$, $|E(G)|\leq {n-3\choose2}$+3, a contradiction. So, we know that either 
$v$ has a neighbor in $F'$ or $wv\in E(G)$. By Corollary \ref{cor2.6}, we have that $pc(G)=2$. The proof is thus complete.
\end{proof}
\begin{theorem}
Let $G$ be a connected graph of order $n\ge 15$. If $|E(G)|\ge {n-4\choose 2}+5$, then $pc(G)\le 3$.
\end{theorem}
\begin{proof}
If $G$ is $2$-edge connected, then $pc(G)\le 3$ clearly holds from Lemma \ref{th2.1}. If $\delta (G)=1$, let $H=G-v$, where $v$ is a pendent vertex. 
Then, $H$ has $n-1$ vertices and $|E(H)|\ge {(n-1)-3\choose 2}+4$. From Theorem \ref{th5.4}, we know that $pc(H)=2$, and so $pc(G)\le 3$. 
In the following, we only consider the graphs with cut-edges and without pendent vertices. Let $e$ be a cut-edge of $G$, and let $G_1, G_2$ be the 
two connected components of $G-e$ with $|V(G_1)|\le |V(G_2)|$. If $|V(G_1)|\ge 5$, we know that $E(G)\le {n-5\choose 2}+11<{n-4\choose 2}+5$, a contradiction. 
So, we know that $|V(G_1)|\le 4$. Since $G$ has no pendent vertices, we know that $|V(G_1)|\ge 3$. Hence, $G_1$ has three or four vertices with at most 
one pendent vertex in $G_1$. It can be easily checked that $pc(G_1)\le 2$. We claim that $pc(G_2)\le 2$. In fact, if $|V(G_1)|=3$, then $G_2$ has $n-3$ vertices 
and $|E(G_2)|\ge {n-4\choose 2}+5-4={(n-3)-1\choose 2}+1$. If $|V(G_1)|=4$, then $G_2$ has $n-4$ vertices and $|E(G_2)|\ge {n-4\choose 2}+5-7={n-4\choose 2}-2$. 
In both cases, we know that $pc(G_2)\le 2$. Consequently, we can easily get that $pc(G)\le 3$.
\end{proof}

\end{document}